\documentclass[12pt]{amsart}

\title{Riordan Groups in Higher Dimensions}
\author{Anthony G. O'Farrell}
\address{Mathematics and Statistics\\Maynooth University\\Co Kildare\\W23 HW31\\Ireland}
\email{anthony.ofarrell@mu.ie}

\RequirePackage{amsfonts, amssymb, amsmath, datetime, url}

\usepackage[T1,T2A]{fontenc}

\date{\today:\currenttime}


\newcommand{\Aut}{\textup{Aut}}
\newcommand{\diag}{\textup{diag}}
\newcommand{\End}{\textup{End}}
\newcommand{\GL}{\textup{GL}}
\newcommand{\gl}{\textup{gl}}

\newcommand{\hcf}{\textup{hcf}}
\newcommand{\Hol}{\textup{Hol}}
\newcommand{\HOT}{\textup{\small HOT}}
\newcommand{\iso}{\textup{iso}}
\newcommand{\Kalg}{{K\textup{-alg}}}
\newcommand{\Kmod}{{K\textup{-mod}}}
\newcommand{\op}{\textup{op}}
\newcommand{\spt}{\textup{spt}}
\newcommand{\sinf}{\textup{sinf}}

\newcommand{\ONE}{1\hskip-4pt1}

\newcommand{\A}{\mathcal{A}}
\newcommand{\AP}{\mathfrak{A}}
\newcommand{\N}{\mathbb{N}}
\newcommand{\Z}{\mathbb{Z}}

\newcommand{\R}{\mathbb{R}}
\renewcommand{\C}{\mathbb{C}}

\newcommand{\EAP}{\widetilde{\mathfrak{A}}}
\newcommand{\ELA}{\widetilde{\mathfrak{L}}}
\newcommand{\F}{\mathcal{F}}
\newcommand{\G}{\mathcal{G}}
\newcommand{\K}{\mathcal{K}}
\newcommand{\M}{\mathcal{M}}
\newcommand{\TI}{\mathcal{T}}
\newcommand{\V}{\mathcal{L}}
\newcommand{\RI}{\mathcal{R}}
\newcommand{\RS}{\widetilde{\mathcal{R}}}
\newcommand{\RV}{{\mathcal{V}}}
\newcommand{\RVV}{{\V^\times\ltimes\K\Sigma}}
\newcommand{\SH}{\mathfrak{S}}
\newcommand{\LA}{\mathfrak{L}}
\newcommand{\hatS}{\widehat{S}}

\newcommand{\pw}[1]{^{\circ #1}}

\newcommand{\ignore}[1]{}

\newtheorem{definition}{Definition}[section]
\newtheorem{theorem}{Theorem}
\newtheorem{lemma}{Lemma}[section]
\newtheorem{corollary}{Corollary}[theorem]
\newtheorem{proposition}{Proposition}[section]

\begin{document}

\begin{abstract}The classical Riordan groups 
associated to a given commutative ring
are groups of infinite matrices (called Riordan arrays)  
associated to pairs of formal power series in one variable.
The Fundamental Theorem of Riordan Arrays relates
matrix multiplication to two group actions on
such series, namely formal (convolution) multiplication
and formal composition. 
We define the analogous Riordan groups involving formal power series
in several variables, and establish the analogue
of the Fundamental Theorem in that context. We discuss related
groups of Laurent series and pose some questions.
\end{abstract}
  
\maketitle

\section{Introduction}

The relation between the composition of formal
series in one variable and matrix multiplication was indentified
at least as early as Eri Jabotinsky's paper 
\cite{Jabotinsky1} from
1947. It has been much exploited.  
The original Riordan group was introduced by 
Shapiro,  Getu, Woan and Woodson in 1991 \cite{ShapiroGWW} 
(see also \cite{Shapiro, Barry}) 
and named in honour of John F. Riordan.  Its elements are
infinite matrices with integral entries, and it is an
important tool in combinatorics. Riordan arrays and groups
with rational, real, complex and other kinds
of entries have been studied.

In this paper we explain how to make the corresponding
groups in several variables.
The several-variable Riordan group and Riordan semigroup
defined below are very closely related to structures
studied by 
Luis Verde-Star \cite{Verde-Star} and by Cheon at al \cite{CHK}.  
The main results give a several-variable analogue
of the Fundamental Theorem on Riordan Arrays,
and we make some remarks about reversibility in these groups.
(An element $g$ of a group $G$ is
\textit{reversible} \cite{OFS} if it is conjugate to its inverse,
i.e. there exists $h\in G$ such that $h^{-1}gh=g^{-1}$.)
There has been
progress on classifying reversibles in one dimension, particularly by Luzon, Moron and Prieto-Martinez \cite{LMP}.
Significant and interesting problems remain, even 
in one dimension.

\subsection{Holomorphs}
Riordan groups are constructed using a special case of a
general procedure.  The simplest case of this procedure
is the construction of the \text{holomorph} of a group:
$$ G \mapsto \Hol(G):= G\ltimes \Aut(G), $$
where $\Aut(G)$ is the group of group automorphisms
of $G$.
This semidirect product 
is the set $G\times\Aut(G)$, equipped with the
multiplication defined by
$$(f,a)(g,b):= (f a(g), a\circ b). $$

More generally, we have this:

For any category $C$, and object $A\in C$, we have
the group $\Aut_C(A)$, consisting of the invertible
arrows (or $C$-morphisms) from $A$ to $A$.

\begin{definition} 
Let $C$ be a category 
equipped with a functor 
$A\mapsto A^\times$ into the category of groups.
	For $A\in C$, the \emph{$C$-holomorph} of $A$ with respect to $^\times$ is
	the group
	$$ \Hol_C(A,^\times):=  A^\times \ltimes \Aut_C(A). $$
\end{definition}

\subsection{Example}	If $C$ is 
	the category of rings with identity
	equipped with the usual functor
	that associates to 
	the ring $R$ the group $R^\times$ of units
	of $R$, then $\Hol_C(R,^\times)$ is the subgroup
of $\Hol(R^\times)$ consisting of those
$(f,a)\in R^\times\times\Aut(R^\times)$
such that $a$ extends to a ring-automorphism of $R$.

\newcommand\Homeo{\textup{Homeo}}
\subsection{Example} If $C$ is the category of topological spaces
and continuous functions, equipped with the 
homotopy functor $X^\times:=\pi_1(X)$, then
$\Aut_C(X)$ is the group $\Homeo(X)$ of bicontinuous
bijections of $X$, and $\Hol_C(X,^\times)$ is the semidirect
product
$$ \pi_1(X) \ltimes \Homeo(X). $$

\subsection{Example} The original holomorph, defined
on the category of groups
is the special case
$$ \Hol(G) = \Hol_\textup{Group}(G,\times) $$
when $\times$ is the identity map on the category, i.e.  $G^\times=G$
for each group $G$.

\ignore{
	$\Hol_C$ is not, in general, a functor from $C$ to the category
	of groups. However, if we define
	$C^\iso$ to be the category having the same objects
	as $C$ but only $C$-invertible arrows $A\to B$
	as arrows, then \\
	(1) $C\mapsto C^\iso$ is a functor on the category
	of categories,\\
	(2) the identity map $C \to C$ is a functor
	from $C^\iso$ to $C$,
	and\\
	(3) the restriction $\Hol_{C^\iso}$ of $\Hol_C$ to
	$C^\iso$ is a functor from $C^\iso$ to the category of groups.
	}

\subsection{The Riordan group}
Fix $K$, an integral domain (commutative, with identity), and consider the
category $\Kalg$ of $K$-algebras with $K$-algebra homomorphisms.
The formal power series rings $\F_d:= K[[x_1,\cdots,x_d]]$ in
$d$ commuting variables are objects of this category.  We refer to
\cite[Section 1]{OF-CFM} for basic notation and basic properties of 
these algebras of power series.  We normally write
a series $f\in\F$ as
$$ f = \sum_{m\in S} f_m m, $$
where $S$ is the semigroup of monic monomials in
$d$ variables. Multiplication in the algebra $\F_d$
is convolution:
$$ (fg)_m = \sum_{p|m} f_p g_{m/p}, $$
where $p|m$ means that the monomial $p$ is a factor of
the monomial $m$. The algebra $\F_d$ is a 
commutative integral domain with identity, and
the multiplicative group
$\F^\times_d$ consists of those series
$f$ having `constant term' $f_1\in\K^\times$.

Note that a $K$-algebra automorphism of $\F$ 
preserves the identity series.

\begin{definition}
	For $d\in\N$, we define \emph{the Riordan group
	in $d$ dimensions over $K$} to be
	$$\RI(K):= \RI_d(K):= \Hol_{\Kalg}(\F_d,^\times). $$
\end{definition}

In \cite{OF-CFM} we described the group
$\G=\G_d$ of formal maps of $K^d$ that fix $0$,
with the operation of formal composition.
We denote its identity by $\ONE$.

For $g\in\G$, we define the \emph{composition
map} $C_g:\F\to\F$ by
$$  
C_g(f):= f\circ g,\ \forall f\in\F. $$
The connection to $\RI_d(K)$ arises from the 
following well-known fact:

\begin{theorem}\label{T:Aut-F}
	The map $\phi:g\mapsto C_g$
	is a group isomorphism from
	$\G_d^{\op}$ onto $\Aut_\Kalg(\F_d)$.
\end{theorem}
\begin{proof}
	This is straightforward. We have
	$$ \phi(g\circ h)(f) = f\circ g\circ h = (\phi(h)\circ\phi(g))(f),$$
	whenever $g,h\in\G$ and $f\in\F$, so $\phi$ is a homomorphism
	from $\G^\op\to \Aut_\Kalg(\F)$.
	If $\phi(g)=\phi(h)$, then for $1\le j\le d$ 
	we have
	$$ g_j = \phi(g)(x_j) = \phi(h)(x_j) = h_j, $$
	so the $j$-th components of the maps $g$ and $h$ coicide.
	Thus $\phi$ is injective.
	Finally, given any $K$-algebra automorphism
	$\alpha$ of $\F$, define
	$$ \psi(\alpha):= (\alpha(x_1),\ldots,\alpha(x_d))\in\G.$$
	Then for each $f\in\F$, we claim that 
	$$ \phi(\psi(\alpha))(f)
	= f\circ \psi(\alpha)
	=\alpha(f).$$
	This tells us that $\phi$ is surjective.
To see the claim, the only point that needs explanation
	is the second equation.  Fix $f\in\F$. 
	For $k\in\N$, we can write
	$f=h + r$, where $h\in K[x_1,\ldots,x_d]$ 
	is a polynomial of degree at most $k$
	and $r\in\M^{k+1}$, the $(k+1)$-st power of the
maximal ideal $\M=x_1\F+\cdots+x_d\F$ of $\F$.
	Since the automorphism $\alpha$ maps $\M^{k+1}$
	into itself, we have
	$$\alpha(f)-\alpha(h)=\alpha(r) \in\M^{k+1}. $$
	Also $\alpha(h)=h\circ\psi(\alpha)$, and
	$$ f\circ\psi(\alpha) - h\circ\psi(\alpha) 
	= r\circ\psi(\alpha)\in \M^{k+1}.$$
	Thus
	$$ f\circ \psi(\alpha)
	-\alpha(f) \in \bigcap_{k=1}^\infty \M^{k+1}=(0), $$
	and the claim holds.
\end{proof}

In the case $d=1$, where $\F:=K[[x]]$ is the algebra
of all formal power series in a single indeterminate $x$
having coefficients in $K$, the abelian group
$\F^\times$ is the set of series of the form
$$ f=f_1+f_xx+ f_{x^2}x^2+\HOT $$
with $f_1\in K^\times$. (We
use $\HOT$ to stand for 
\lq\lq higher-order terms''. In this case $\HOT$ stands for
$\sum_{n=3}^\infty f_{x^n}x^n$.)
For this case, the theorem says that the group
$\Aut_\Kalg(\F)$ is isomorphic to  the group $(\G,\circ^\op)$ 
of formally-invertible 
power series of the form
$$ g= g_xx+ g_{x^2}x^2+\HOT $$
with 
$g_x\in K^\times$, with the action given by
$$ 
\begin{array}{rcl}
	g(f) &=&
	f\circ g = f_1 +f_x g(x) + f_{x^2}g(x)^2 +\HOT
	\\
	&=& f_1 + (f_xg_x)x + (f_xg_{x^2}+f_{x^2}g_x^2)x^2+\HOT.
	\\
\end{array}
	$$
and the group multiplication given by
$$ g\circ^\op g' = g'\circ g. $$

\begin{corollary}
Let $d\in\N$ and $K$ be an integral domain with identity.
	Then the  $d$-dimensional Riordan group $\RI_d(K)$ 
	is isomorphic to the semidirect product
$$ \F_d^\times \ltimes \G_d^{\op}, $$
with the multiplication
$$ (f,g)(f',g') := (f\cdot (f'\circ g), g'\circ g), $$
for $f,f'\in\F_d$ and $g,g'\in\G_d$.
\end{corollary}
The original Riordan group corresponds to the case $d=1$ and
$K=\Z$.

Throughout the rest of this paper we identify $\RI_d(K)$ and 
$ \F_d^\times \ltimes \G_d^{\op} $, and
represent elements of 
$\RI_d(K)$ as pairs $(f,g)\subset\F_d^\times\times\G_d$. 

We define
$$ \LA:= \{(1,g): g\in\Aut(\F) \},\quad \AP:= \{(f,\ONE): f\in\F^\times
\}.$$
These are subgroups of $\RI$, and in analogy with the case $d=1$
we call $\LA$ the \emph{Lagrange subgroup} and $\AP$
the \emph{Appell subgroup} (cf. \cite{Barry}).  Clearly, $\LA$ is isomorphic to
$\G^{\textup{op}}$ and $\AP$ is isomorphic to $\F^\times$.

Many researchers working on one-dimensional Riordan groups over $\Z$ or $\C$ 
confine attention to $(f,g)\in\F_1^\times\times\G_1$
such that $f_1=1$ and $g=x+\HOT$. In this, they follow the lead of Shapiro
\cite{Shapiro}. So for general $d$ and $K$ we define 
\emph{the Shapiro subgroup} to be
$$ \SH:= \{(f,g)\in \RI: f_1=1 \textup{ and } g \in\TI)\}, $$
where
$$ \TI:= \{g\in \G: g=\ONE+\HOT\} $$
is the subroup of $\G$ consisting of the formal maps
\emph{tangent to the identity}, i.e. maps in the kernel
of the homomorphism $L:\G\to\GL(K,d)$ that takes 
a map $G$ to the linear term $L_1$ in its expansion
$$ g = \sum_{k=1}^\infty L_k  $$
in homogeneous polynomials of degree $k=1$,$2$,$\ldots$.

\section{Indexed matrices}
The original description of Riordan arrays, ably described in 
the undergraduate textbook \cite[Chapter 1]{Barry}, is neither
in terms of $\Hol_\Kalg(\F,^\times)$ nor $\F^\times\ltimes\G^\op$.
The arrays are lower-triangular infinite matrices
with integral entries, in which each row is derived
from the previous one by an unchanging rule that involves
linear combinations.  The other descriptions exploit
Jabotinski's idea, compressing the information in
the array into a pair of formal power series.
Well-known examples include (1) the 
array whose rows
are those of Pascal's triangle, corresponding to the
series of $\frac1{1-x}\in\F_1^\times$ and 
$\frac{x}{1-x}\in\G_1$, and (2) the 
Riordan-group inverse to the array corresponding to the 
simple finite series $1-x\in\F$ and $x-x^2\in\G$,
corresponding to 
$$ 
\frac{1-\sqrt{1-4x}}{2x}\in\F^\times
\textup{ and }
\frac{1-\sqrt{1-4x}}{2}\in\G,
$$
whose first column lists
the Catalan numbers (first named by Riordan). 

We proceed to discuss the way to extend the matrix representation
of the one-dimensional Riordan groups to dimensions $d>1$.

\subsection{Matrix structure}
Let $K$ be an integral domain with identity, and $n\in\N$. We are used to
$\gl(n,K)$, the usual $K$-algebra (i.e. ring and $K$-module)
of $n\times n$ 
matrices over $K$. The typical entry in a matrix $(a_{ij})\in\gl(n,K)$
is indexed by two natural numbers $i$ and $j$ in the set
$\{1,2,\ldots,n\}$.  We tend to think of a matrix as an 
structured object, whose structure is derived from the
ordered structure of $\N$.  For some matrix applications
the arithmetical structure of $\N$ is significant; examples
are those using Hankel and Toeplitz matrices.  For others, the order
structure of $\N$ is significant; examples are those using upper-triangular
matrices. But sometimes neither the arithmetical or ordered
structure of $\N$ is important; the natural numbers are just used
as labels.  
 
Letting $S=\{1,\ldots,n\}$, the set $K^S$
is a $K$-module, and the set $\End_K(K^S)$
of all $K$-module endomorphisms of $K^S$
has a natural $K$-algebra structure,
where the addition and scalar multiplication
are pointwise and the multiplication
$$ \End(K^S) \times \End(K^S) \to \End(K^S) $$
is composition.  Let $e_j(i):= 1$ if $i=j$ and $0$ otherwise.
The map that sends the endomorphism
$\phi\in \End(K^S)$ to the matrix 
$(a_{ij})\in \gl(n,K)$ defined by
\begin{equation}\label{E:1}
 a_{ij}:= \phi(e_j)(i),\ \forall i,j\in S
\end{equation}
is bijective, and preserves the $K$-algebra
structures, so this gives another way to view
$\gl(n,K)$.  

But $\End(K^S)$ makes sense, and 
has the same $K$-algebra structure for \emph{arbitrary
sets} $S$, and the same formula  (\ref{E:1})
defines an object that we can think of as
a matrix, i.e. a two-dimensional array, in which
the rows and columns are indexed by elements
of $S$.  Let us denote the set of all these
matrices by $\gl'(S,K)$. 
If $S$ is infinite,
it is no longer true that the map from endomorphisms
to matrices is injective, because the matrix
only depends
on the values of the endomorphism
on the linear
span of the $e_j$, a proper submodule $F$ of $K^S$.
When the algebraic operations 
of pointwise addition and scalar multiplication on
$\End(K^S)$ are transferred to $\gl'(S,K)$,
the formulas look familiar:
$$\begin{array}{rcl}
(a_{ij}) + (b_{ij})&:=& (a_{ij}+b_{ij}),\\
c\cdot(a_{ij}) &:=& (ca_{ij}),
\end{array}
$$
However, since an endomorphism $\phi$ may map
an $e_j$ outside the subspace $F$, there is not
enough information in the matrix associated to
$\phi$ to allow us to calculate the matrix
associated to $\psi\circ\phi$, for another
endomorphism $\psi$.  We can overcome this by
restricting attention to $\End_K(F)$, and then
composition transfers to the familiar-looking rule: 
$$\begin{array}{rcl}
(a_{ij}) \times (b_{ij})&:=& (c_{ij}),\\
	\textup{where }c_{ij} &=& \sum_{k\in S} a_{ik}b_{kj}. 
\end{array}
$$
This rule actually gives a well-defined product on 
matrices such that, for each $i\in S$, the
entry $a_{ij}$ is zero for all but a
finite number of $j\in S$, and we denote
this set of matrices by $\gl(S,K)$. (Note that 
this set is somewhat larger than the image of
$F$.)  With these operations, this
set is a $K$-algebra, for arbitrary sets $S$.

This simple idea is useful, because
it is sometimes more natural to index the rows and columns of a matrix
by elements of a particular finite or infinite set,
and it may be artificial (or impossible)
to label these using natural numbers.
For our immediate purposes, the natural
index set that presents itself is the set
$S$ of monic monomials in $d$ variables \cite{OF-CFM},
because the semigroup structure of that $S$ is relevant
to our purpose.

$\gl(S,K)$ is a semigroup with identity under multiplication; the
identity matrix $I$ has $I_{ii}=1$ and $I_{ij}=0$ when $i\not=j$.
We denote the group of invertible elements by $\GL(S,K)$.


\section{Matrix Representation}
Let $K$ be an integral domain with
identity, $d\in\N$ and $(f,g)\in\RI=\RI_d(K)$, and let $S$ be the
set of monic monomials in $d$ variables.  Then 
we can define an associated matrix 
$M(f,g)\in \GL(S,K)$ by setting the $(m,n)$ entry equal to
the coefficient of the monomial $m$ in $f\cdot (n\circ g)$.
This means that you take the composition of the monomial
$n\in S$ with the formal map $g\in\G$, getting an element
of $\F$, you then multiply this in $\F$ by the formal
series $f$, and then you take the coefficient of the 
monomial $m\in S$.  Expanding the product, this gives
$$ M(f,g)_{m,n} = \sum_p f_p (p \cdot (n\circ g))_m
=  \sum_{p|m} f_p (n\circ g)_{m/p}. $$

\begin{theorem}\label{T:1}
 The map $M$ is an injective group homomorphism
from the Riordan group $\RI$ into $\GL(S,K)$.  
\end{theorem}

\begin{proof}
First, we show that 
\begin{equation}\label{E:2}
 M( (f,g)(f',g') ) = M(f,g) M(f',g') 
\end{equation}
whenever $f,f'\in\F^\times$ and $g,g'\in\G$. 

Let $(f,g)$ and $(f',g')$ belong to $\RI$. 
Fix $m,n\in S$.  Then the $(m,n)$ entry in
$$ M( (f,g)(f',g') ) = M( f\cdot (f'\circ g), g'\circ g) $$
is the coefficient of $m$ in
$$ f\cdot (f'\circ g) \cdot n(g'\circ g). $$
We compare this with the $(m,n)$ entry in
the matrix product\\
 $M(f,g)M(f',g')$, which is
$$ \sum_{p\in S} (f\cdot p\circ g)_m \cdot (f'\cdot n\circ g')_p.$$
 
To begin with
we look at three special cases.

\smallskip
Case 1: $(f,\ONE)(f',g')$, i.e. the first factor is in
the Appell subgroup.

We have to compare the coefficient of $m$ in $f\cdot f'\cdot (n\circ g')$ 
with the $(m,n)$ entry of the matrix product $M(f,\ONE)M(f',g')$. 
The latter entry is
$$ \sum_p (f\cdot p)_m (f'\cdot(n\circ g'))_p = 
\sum_{p|m} (f\cdot p)_m (f'\cdot(n\circ g'))_p,
$$
because $(f\cdot p)_m$ is obviously zero unless $p$ divides $m$.
But $(f\cdot p)_m = f_{m/p}$, so the entry is 
$$ \sum_{p|m} f_{p/m}\cdot (f'\cdot(n\circ g'))_p, $$
and this is exactly the coefficient of $m$ in 
$f\cdot f'\cdot n\circ g'$.  So Equation (\ref{E:2}) holds in this case.

\smallskip
Case 2: $(f,g)(1,g')$, i.e. the second factor is in the
Lagrange subgroup.

We have to compare the coefficient of $m$ in $f\cdot (n\circ g'\circ g)$,
which equals
$$ \sum_{p|m} f_{m/p} (n\circ g'\circ g)_p, $$ 
with the $(m,n)$ entry of $M(f,g)M(1,g')$.  This entry is
$$ \sum_p (f\cdot (p\circ g))_m (n\circ g')_p
= \sum_p \sum_{q|m} f_{m/q} (p\circ g)_q (n\circ g')_p. $$
Interchanging the order of summation, this equals
$$ 
\sum_{q|m} f_{m/q} \sum_p (p\circ g)_q (n\circ g')_p. $$
But the inner sum is exactly the coefficient of $q$
in $(n\circ g')\circ g$, so the entry equals
$$\sum_{q|m} f_{m/q} (n\circ g'\circ g)_q. $$
Replacing the dummy variable $q$ by $p$, we see that
Equation (\ref{E:2}) also holds in this case.

\smallskip
Case 3: $(1,g)(f',1)$, i.e. the first factor is in
the Lagrange subgroup and the second in the Appell subgroup.

The $(m,n)$ entry in
$$ M( (1,g)(f',\ONE) ) = M( f'\circ g, g) $$
is the coefficient of $m$ in 
$ (f'\circ g) \cdot (n\circ g)$, which equals 
$$ \sum_{p|m} (f'\circ g)_p \cdot (n\circ g)_{m/p}.$$
Now $f'\circ g = \sum_q f'_q\cdot(q\circ g)$, so
this matrix entry is 
$$ \sum_{p|m}\sum_q f'_q\cdot(q\circ g)_{m/p}\cdot(n\circ g)_p 
=
 \sum_q f'_q \sum_{p|m} (q\circ g)_{m/p}\cdot(n\circ g)_p. 
$$
Now $\sum_{p|m} (q\circ g)_{m/p}\cdot(n\circ g)_p$
is the coefficient of $m$ in $(q\circ g)\cdot(n\circ g) = (q\cdot n)\circ g$,
so the entry is equal to 
$$ \sum_q f'_q \cdot ((nq)\circ g)_m = \sum_{n|p} f'_{p/n}\cdot (p\circ g)_m.$$
On the other hand, the $(m,n)$ entry in $M(1,g)M(f',\ONE)$ is
$$ \sum_p (p\circ g)_m \cdot (f'\cdot n)_p =
\sum_{n|p} (p\circ g)_m \cdot f'_{p/n},$$
since $(f'\cdot n)_p$ is zero unless $n$ is a factor of $p$.
Thus Equation (\ref{E:2}) holds in this case also.

\smallskip
In the general case, for 
$f,f'\in\F^\times$ and
$g,g'\in\G$, we can write
$$(f,g)(f',g')= (f,\ONE)(1,g)(f',\ONE)(1,g'),$$
 so by Case 1,
$$ M( (f,g)(f',g') ) = M(f,\ONE) M((1,g)(f',\ONE)(1,g')),$$
and by Cases 2 and 3 this is
$$ M(f,\ONE) M((1,g)(f',\ONE)) M(1,g'))
= M(f,\ONE) M(1,g) M(f',\ONE) M(1,g'))
.$$
Applying Case 1 twice more, this equals
$M(f,g) M(f',g')$, so we have shown that 
Equation (2) holds for each
$f,f'\in\F^\times$ and
$g,g'\in\G$.

\medskip
It follows that each $M(f,g)$ is invertible, i.e. lies in $\GL(S,K)$
(and not just in $\gl(S,K)$),  
because it is inverted by $M((f,g)^{-1})$. Here we use the
facts that $\RI$ is a group, that $M$ maps products
in $\RI$ to products in $\gl(S,K)$, and that $M$ maps
the identity $(1,\ONE)$ of $\RI$ to the multiplicative
identity matrix $I$ in the ring $\gl(S,K)$. 

\medskip
Finally, we prove that $M$ is injective:
Suppose $M(f,g) = M(f',g')$.  

Let $m\in S$.
The $(m,1)$ entry of $M(f,g)$ is the
coefficient of $m$ in $f\cdot1=f$. Thus $f_m=f'_m$. 
Since this holds for all $m\in S$, we have $f=f'$.
Since $M$ is a group homomorphism, and $(f,g)=(f,\ONE)(1,g)$, 
$$ M(1,g) = M(f,\ONE)^{-1} M(f,g) =
M(f',\ONE)^{-1}M(f',g') = M(1,g'). $$
For $1\le j\le d$, the $(m,x_j)$ entry of $M(1,g)$
is the coefficient of $m$ in $x_j\circ g$, which is just the
$j$-th component of $g$. Thus $g$ and $g'$ have the same
coefficients in each component, hence $g=g'$.
\end{proof}

\section{Fundamental Theorem on Riordan Arrays}
We continue to assume $d\in\N$ and to denote by $S$
the semigroup of all monic monomials in $d$ variables.

The formula
\begin{equation}\label{E:3}
	M(f,g)_{mn} := (f\cdot(n\circ g))_m, \ \forall m,n\in S 
\end{equation}
defines a matrix $M(f,g)$, an element of $\gl(S,K)$,
whenever $f\in\F$ and $g\in \F^d$, and not just
for $f\in\F^\times$ and $g\in\G$.  The matrix
$M(f,g)$ belongs to $\gl(S,K)$ because
its $(m,n)$ entry is zero as soon as
the degree of $n$ exceeds the degree of $m$.

The right-hand side
of the formula
\begin{equation}\label{E:4}
 (f,g)(f',g'):= (f\cdot f'\circ g, g'\circ g) 
\end{equation}
does not make sense in this generality, but it does
make sense when $f,f'\in\F$ and
$g,g'\in \M^d$.  The right-hand side belongs
to the cartesian product $\F\times\M^d$. 

\begin{definition}  We define  the \emph{Riordan semigroup}
$\RS$ to be the cartesian product $\F\times\M^d$, equipped with
the product defined by Equation (\ref{E:4}).
\end{definition}

It is straightforward to check that the multiplication is
(still) associative, so we do indeed have a semigroup here.
The identity $(1,\ONE)$ of $\G$ is still the identity
for $\RS$.  The subset $\RI\subset\RS$ is precisely the
subgroup of invertible elements of $\RS$.

With its usual ring multiplication, the matrix algebra
$\gl(S,K)$ is a semigroup with identity.    

\begin{theorem}\label{T:2}
The map $M$ is a semigroup homomorphism
from $\RS$ into $\gl(S,K)$.
\end{theorem}

\begin{proof}
This amounts to saying that Equation (\ref{E:2})
holds for general $(f,g), (f',g')\in\RS$.
Inspection of the proof of Theorem \ref{T:1}
reveals that we did not make any use of the 
invertibility of $(f,g)$ or $(f',g')$
in proving Equation (\ref{E:2}). We did
use, at the last step, the factorization
$$ (f,g) = (f,\ONE)(1,g), $$
but this also holds for general $f\in\F$ and
$g\in\M^d$.  Thus the present theorem may
be viewed as a corollary of the proof
of Theorem \ref{T:1}.
\end{proof}

Theorem \ref{T:2} is an extension to
several variables of the so-called Fundamental
Theorem on Riordan Arrays (FTRA). 
The following immediate corollary
may be recognised as alternative
formulations of FTRA:

\begin{corollary}\label{C:1}
For $f,u\in\F$ and $g\in\G$, we have
$$ (f\cdot (u\circ g) )_m 
	= \sum_{\deg p \le \deg m} u_p\cdot (f\cdot(p\circ g))_m $$
whenever $m\in S$.
\end{corollary}
\begin{proof}
Apply Theorem \ref{T:2} to the $(m,1)$ components
of both sides of Equation (\ref{E:2}), and
replace $(f,g)$ by $(f,\ONE)$ and $(f',g')$
by $(u,g)$. This gives
$$ (f\cdot (u\circ g) )_m 
	= \sum_p  (f\cdot(p\circ g))_m \cdot u_p.$$
Now note that $p\circ g$ has nonzero 
$m$-coefficient only when the degree of
$p$ is at most the degree of $m$, and
the result follows.
\end{proof}

\subsection{}
We remark that the theorem shows that
the equation in  Corollary \ref{C:1}
also holds more generally, for $g$ belonging
to the semigroup $\M^d$:
\begin{corollary}\label{C:2}
Let $f,u\in\F$ and $g\in\M^d$, $g\not=0$, and suppose that
a lowest-degree monomial $n$ with a nonzero $(g_j)_n$
for some $j\in\{1,\ldots,d\}$ has degree $k$.
Then $k\ge1$ and 
$$ (f\cdot (u\circ g) )_m 
	= \sum_{\deg p \le (\deg m)/k} u_p\cdot (f\cdot(p\circ g))_m $$
whenever $m\in S$.
\end{corollary}
\begin{proof}
The point here is that $(p\circ g)_m$ is zero
if $k\cdot\deg p>\deg m$.
\end{proof}

\section{Laurent Series and Verde-Star structures}
The rest of this paper was inspired by a lecture
of Luis Verde-Star, given in Madrid in 2017,
and describing  two-way
infinite matrices associated with an
even 
larger algebra than $\F$, on which a subgroup
of $\G$ acts as a group of
automorphisms.  The topic was related to his
paper \cite{Verde-Star}. 

\subsection{Monomials}
We embed the (commutative, cancellation) semigroup $S$ of monic monomials in the
larger group $\hatS$ of objects $x^i$, where
now we allow any $i\in \Z^d$. As a group, $\hatS$
is isomorphic to the additive group
$(\Z^d,+)$, the free abelian group on
$d$ generators, via the
map $x^i\mapsto i$. It has a partial order, defined by
$$ x^i \le x^{i'} \leftrightarrow i_j\le i'_j \ \forall j. $$
(Equivalently, $m\le n$ means that $n/m\in S$.)
It is not hard to see that each subset $A\subset\hatS$
that is bounded below has a unique greatest lower bound.
We denote this lower bound by $\inf A$.

In case $A\subset S$, $\inf A=\hcf A$
(cf. \cite[Subsection 1.1]{OF-CFM}). 

\subsection{Laurent series}
For an arbitrary formal series 
$f(x):= \sum_{m\in\hatS} f_m m$ in $d$ variables,
with coefficients
$f_m\in K$, we define the \emph{support of $f$}
to be the set
$$ \spt f = \{m\in\hatS: f_m\not=0\}. $$

We define the subset of \emph{Laurent series} by
$$ \V:= \left\{f=\sum_{f\in\hatS}f_mm: \spt f \textup{ is bounded below}
\right\}
$$
The set $\V$ becomes a commutative $K$-algebra with identity
when endowed with
term-by-term addition and convolution multiplication
$$ \left(\sum_m f_mm\right) \left(\sum_m f'_mm\right):=
\sum_m \left( \sum_p f_p\cdot f_{m/p}\right) m.
$$
For nonzero $f\in\V$, we define the \emph{vertex of $f$}
to be
$$ v(f):= \inf\spt(f). $$
This coincides with the definition previously given
in \cite{OF-CFM},
in case $f\in\F$.  

We refer to $\V$ as the \emph{Laurent series
algebra} associated to the integral domain $K$.
It contains $\F$ as a subalgebra. 

We refer to the group $\V^\times$ as \emph{the Laurent 
series group of $K$}.
We note that $\hatS$ is a subgroup of $\V^\times$,
i.e. the product of two monomials $mm'$ in $\V$
is their product in $\hatS$.

For $f\in\V$, we always denote the coefficient
of the monomial $m$ in the series for $f$ by $f_m$.
Notice that $(n\cdot f)_m=f_{m/n}$ whenever
$m,n\in\hatS$.

\subsection{Example}
For a set $A\subset\hatS$ that is bounded below,
let $\chi_A$ denote the series $\sum_{m\in A}m$.
Then $\chi_A\in\V$.  If $B\subset\hatS$ is another
such set, then the product $h:=\chi_A\chi_B$ has
coefficients $h_m$ that count the number of ways
to factor $m$ as $pq$, with $p\in A$ and $q\in B$.

Thus if $f,f'\in\V$, and we let $A:=\spt(f)$ and
$B:=\spt(f')$, then the coefficient $h_m$
measures the complexity of computing $(ff')_m$.

\subsection{Other options}
There are a good many other ways to construct algebras
that can be called Laurent series algebras. The 
algebra $\V$ is just one possibility.

The Laurent series algebra $\V$ is a (rather small)
subalgebra of the quotient field $\widehat{\F}$ of $\F$. Most
series $f\in\F$ are noninvertible in $\V$.  For
instance, in two dimensions the element
$(x_1+x_2)^{-1}\in\widehat{F}$ does not belong to
$\V$ (see below).

It is possible to describe $\V$ in terms of fields
of fractions of power series rings in one variable,
as follows.  Let $\Sigma:=\Sigma_d$ denote the symmetric group
on $\{1,\ldots,d\}$. Then for $\tau\in\Sigma$ we can
form, in turn, (1) the (formal) power series ring $K[[x_{\tau(1)}]]$,
(2) its field of fractions $K((x_{\tau(1)}))$,
(Laurent series in $x_{\tau(1)}$ over $K$),
(3) the power series ring 
$$K((x_{\tau(1)}))[[x_{\tau(2)}]], $$
(4) its field of fractions
$$K((x_{\tau(1)}))((x_{\tau(2)})), $$
(Laurent series in $x_{\tau(2)}$ with coefficients
that are Laurent series in $x_{\tau(1)}$),
and so on, ending up with the field
(2d) of Laurent series 
$$ \V_\tau:= K((x_{\tau(1)}))((x_{\tau(2)}))
\cdots((x_{\tau(d)})). $$
This field $\V_\tau$ is larger than $\V$ once $d>1$, because
an element may have nonzero coefficients 
on monomials involving infinitely many different negative powers
of $x_{\tau(1)}$.  We have, in fact,
$$ \V = \bigcap\{ \V_\tau: \tau\in\Sigma_d \}. $$

\medskip
The paper of Monforte and Kauers
\cite{Monforte-Kauers} gives a useful historical survey
of the various other kinds of Laurent series
algebras one might consider, with particular emphasis
on those that are fields. It appears that such
algebras are useful in applications to combinatorial
problems such as counting lattice-point
in polytopes and counting integer partitions.
They also arose, as far back as the late 1940's, in 
work of A.I.~Malcev and B.H.~Neumann
on embedding group algebras into
a division algebra \cite{Malcev, Neumann}, and indeed germs of the
subject are found in work of Hans Hahn from forty years before that.
There is a substantial literature on
\textit{Malcev-Neumann fields} of power series,
related to totally-ordered groups.
If we exclude consideration of
nonintegral exponents, then the \textit{ne plus ultra},
when it comes to creating a multiplicatively-closed
space of formal series $f=\sum_{m\in\hatS}f_mm$, is to
confine attention to $f$ whose supports lie in
a sub-semigroup of $\hatS$.  This leads one to consider
total orders on $\Z^d$, and the intersection
of $\Z^d$ with cones in $\R^d$.  There are in fact
uncountably many choices, each leading to a different
field of ``Laurent series''.  None of
these contains or is contained in the 
relatively-simpleminded algebra $\V$ that we consider,
which depends on the standard (non-total) lattice
partial order on $\Z^d$, 
but many of the elementary considerations
we meet have, of course, arisen in these previous
investigations, and we make no claim to
originality for our observations about them,
but
as far as we know, the two theorems below are new.

Verde-Star considered a family of partial orders
on $\Z^d$ indexed by nonegative integers, with associated cones
$C_s$, and algebras of Laurent series (over $\C$)
associated to them. Our algebra $\V$ corresponds to
his $\F_s$ with $s=0$, and our $\F$ corresponds to his
$\F_0$, in that case.

\subsection{The Verde-Star-Riordan group}
\begin{definition}
We define the Verde-Star-Riordan group to
be 
	$$ \RV = \Hol_\Kalg(\V,^\times), $$
\end{definition}
Thus
$$ \RV = \V^\times \ltimes \Aut_\Kalg(\V). $$

\subsection{The group $\V^\times$}
\begin{proposition}\label{P:VS}
Let $f\in\V$. 
(1) If $f\not=0$, then
$h=v(f)^{-1}\cdot f\in\F$, and $v(h)=1$.
\\
(2) $\V$ is an integral domain.
\\
(3) $v(ff')=v(f)v(f)$ whenever $f,f'\in\V$.
\\
(4) $f\in\V^\times$ if and only if $f_{v(f)}\in K^\times$. 
\end{proposition}
\begin{proof}
(1) 
	The fact that $v(f)=\hcf(\spt(f))$ implies that
for each $j\in\{1,\ldots,d\}$ there exists
$p\in\spt(f)$ with $f_p\not=0$ and $p_j=v(f)_j$.
Thus $q:=p/v(f)$ belongs to $\spt(h)$ and has $q_j=1$
(i.e. $q$ `does not involve $x_j$'). Thus
$v(h)=1$.
(cf. \cite[proposition 1.1]{OF-CFM}, which has essentially
	the same proof.)
\\
(2) If $f$ and $f'$ are nonzero elements
of $\V$, then $h:=f/v(f)$ and $h':=f'/v(f')$
are nonzero elements of $\F$, and
if $ff'=0$, then $hh'=0$, contradicting
the fact that $\F$ is an integral domain.
Thus $\V$ is an integral domain.
\\
(3) 
	 We have to show that $v(ff')_j= v(f)_jv(f')_j$,
	for each $j\in\{1,\ldots,d\}\}$.
	We may assume that $j=d$, without loss in generality.
	By the definition
	of the vertex, there must be a nonzero coefficients
	on some term of $f$ of the form $mv(f)_d$, where
	$m\in\hatS$ does not involve $x_d$ at all. In other words,
	if we write $h:=f/v(f)_d$, then 
	$$k: =h(x_1,\ldots,x_{d-1},0) \not =0,$$
	and $k\in\V_{d-1}$. 
	Similarly, letting $h'=f'/v(f')_d$, 
	$$k': =h'(x_1,\ldots,x_{d-1},0) \not =0,$$
	and $k'\in\V_{d-1}$.  Since $\V_{d-1}$ is an integral domain,
	we have $kk'\not=0$, $v(hh')_d=1$, and $v(ff')_d=v(f)_dv(f')_d$.
\\
(4) 
Let $h:= v(f)^{-1}f$.  Then $h\in\F$ and $h_1=f_{v(f)}$.
Also $v(h)=1$.
\\
If $f_{v(f)}\in K^\times$, then $h\in\F^\times$, and
$$ f \cdot ( v(f)^{-1}\cdot h^{-1} ) = (v(f)^{-1}\cdot f)\cdot
h^{-1} = h\cdot h^{-1} = 1, $$
so $f\in\V^\times$.

Conversely, if $f$ has an inverse $k\in\V$, then
$$ 1 = fk = f v(f)^{-1} v(f)k = hv(f)k, $$
so $h\in\V^\times$. 
	Thus it remains, in view of  \cite[proposition 1.3]{OF-CFM}
to show that $h^{-1}\in\F$.
Suppose $h^{-1} = b$. We have to show that $1\le v(b)$.

Suppose, on the contrary, that $1\not\le v(b)$. Then
some $x_j$ appears with a negative power in $v(b)$,
and we may suppose without loss in generality
that it is $x_1$, so that $v(b)_1 = x_1^{-k}$
for some $k\in\N$. The series
$$ a:= \sum_{m\in\hatS, m_1= x_1^{-k}} b_mm
$$ 
is not zero, by definition of $v(b)$.

Similarly, since $v(h)=1$, 
the series
$$ r:= \sum_{m\in S, m_1=1} h_mm $$
is not zero.

Thus, since $hb=1$, 
$$ 0 = \sum_{m\in\hatS, m_1= x_1^{-k}} (hb)_mm
= ar, $$
contradicting the fact that $\V$
is an integral domain.  Thus $b\in\F$, as required.
\end{proof}   

For example, taking $f=x_1+ x_2\in\F_2$, we have
$v(f)=1$ and $f_1=0$, so $f$ is not invertible.
This is probably just as well, for its reciprocal 
seems to have two different formal Laurent series:
$$ \frac1{x_1+x_2} \sim
\sum_{k=0}^\infty x_1^{-k-1}x_2^k
\sim
\sum_{k=0}^\infty x_1^kx_2^{-k-1}.
$$
Of course neither series belongs to $\V$.

From the proof, we note:
\begin{corollary}\label{C:V-times-is-product}
	(1) $\V^\times = \hatS \F^\times$ and $\hatS\cap\F^\times=(1)$,
	i.e. $\V^\times$ is the inner direct product of
	$\hatS$ and $\F^\times$.
\\(2) The restriction $v|\V^\times$
	is a group homomorphism from $\V^\times$ onto $\hatS$.
\\(3)
$\ker v|\V^\times = \{f\in \V^\times: v(f)=1\} = \F^\times$.
\qed
\end{corollary}

\subsection{Composition}
For $g\in\G$, we have seen that
the composition
map $C_g$ is a $K$-algebra automorphism of $\F$.
We cannot define the composition $f\circ g$
for arbitrary $f\in\V$ and $g\in\G$ in a sensible way.
For instance, one would expect the composition
of $x_1^{-1}$ and the map $(x_1,x_2)\mapsto
(x_1+x_2,x_2)$ to represent $(x_1+x_2)^{-1}$, but 
$x_1+x_2$ is not invertible in $\V$.
However, we have the following limited composition:

Given two $d$-tuples\footnote{
Note that there is room for confusion
between $f_1\in\F$, the first component of
an $f\in\F^d$, and $f_1\in K$, the coefficient
of $1_S$ in a series $f\in\F$. In
the first case the subscript $1$ is $1_{\N}$,
and in the second it is $1_S$. It is necessary
to pay attention to the context to avoid this confusion.}
$f:=(f_1,f_2,\ldots,f_d)\in\F^d$ and
$f':=(f'_1,f'_2,\ldots,f'_d)\in\F^d$,
we define the coordinatewise product
$f*f'$ by setting
$$ (f*f')(x) = (f_1f'_1,\ldots, f_df'_d).$$  
For instance,
$$ (x_1,x_2)*(1+x_1+x_2^2, 1-x_2-x_1^2)
= (x_1+x_1^2+x_1x_2^2, x_2-x_2^2-x_1^2x_2).$$
This product $*$ maps $\F^d\times\F^d\to\F^d$, and
is commutative and associative. If we define
addition coordinatewise, then $*$ distributes
over addition, and $\F^d$ becomes a
$K$-algebra.  In fact, it is just the
direct product of $d$ commuting copies of $\F$.
It has the identity $(1,\ldots,1)$
and its group of units is 
$$ (\F^d)^\times = (\F^\times)^d
=\{f\in\F^d: f_j\in\F^\times\ \forall j\}.$$
$\F^d$ is not an integral domain when $d>1$,
since, for instance, 
$$ (1,0)*(0,1)=(0,0)=0_{\F^d}.$$
In case $f$ is a formal map, i.e. $f\in\M^d$, 
the product $f*f'$ also belongs to $\M^d$
for each $f'\in\F^d$.  In the particular case
$f=\ONE$, we use the more suggestive notation
$$x*f':= \ONE*f' = (x_1f'_1,\ldots,x_df'_d).$$

We define the set
$$ \K:= \{ x*f: f\in(\F^d)^\times \}.$$
This is a subgroup of $\G$, because the linear
part
$L( x*f )$ of $x*f$ is represented by the
diagonal matrix $\diag((f_1)_1,(f_2)_1,\ldots,(f_d)_1)$,
and this is invertible in $\gl(d,K)$ 
if each component $f_j\in\F^\times$.  Note that
the group $L(\K)\le\GL(K,d)$ of $K$-linear maps, isomorphic
to $(K^{\times},\times)^d$ may be regarded as 
a subgroup of $\K$, and then $\K$ is the
semidirect product of 
$$\ker L|\K = \{(x_1(1+\HOT),\ldots,x_d(1+\HOT))\}$$
and $K^{\times d}$.  

Moreover, if $g=x*f\in\K$, then 
each $f_j$ belongs to $\F^\times$, and hence
each component $g_j=x_j f_j$ belongs to
$\V^\times$.
This allows us to define $m\circ g\in\V^\times$,
for arbitrary $m\in \hatS$ by writing
$m=x^i$ and defining
$$ m\circ g:= \prod_{j=1}^d (x_j f_j)^{i_j}.$$

Notice that $m\circ g= m\cdot m(f_1,\ldots,f_d)\in m\F^\times$
when $m\in\hatS$ and $g=x*f\in\K$.

We can then define the composition $f\circ g$
for arbitrary $f\in\V$ and $g\in\K$ by
setting
$$ f\circ g:= \sum_{m\in\hatS} f_m m\circ g.$$
this formal Laurent series
actually defines an element of $\V$, because
$m = v(m\circ g)$, so $v(f)=v(f\circ g)$.

We denote the composition map $f\mapsto f\circ g$
on $\V$ by the same symbol $C_g$ as we already 
use for its restriction to $\F$.

\begin{proposition}
	For each $g\in\K$, the composition
	map $C_g:\V\to\V$ is a $K$-algebra
	automorphism of $\V$.
\end{proposition}
\begin{proof}
	Composition preserves sums, products
	and scalar multiples. So $C_g$ is
	a $K$-algebra homomorphism. It is
	inverted by composition with $g^{-1}$,
	so it is an isomorphism.
\end{proof}

\begin{corollary}
	The map $g\mapsto C_g$ is
	an injective homomorphism of
	$\K^\op$ into $\Aut_\Kalg(\V)$.
\end{corollary}
\begin{proof}
	The injectivity follows from the
	fact that $x_j\circ g=g_j$, the
	$j$-th component of $g$.

	The map is a homomorphism because
	of the associativity of composition.
\end{proof}

\subsection{Permutations}
The permutation group $\Sigma:=\Sigma_d$ acts on the 
symbols $x_j$ by permuting the $j$'s, and this
gives another automorphism of $\V$.
This is the same as the composition automorphism
corresponding the the linear isomorphism
of $K^d$ obtained by permuting the entries
of the $d$-tuples, so we also regard $\Sigma$
as a subgroup of $\G$, and note that it intersects
$\K$ only in the identity, and acts on
$\K$ by conjugation. The product set $\K\Sigma_d$
is thus a subgroup of $\G$, isomorphic to the
semidirect product $\K\ltimes\Sigma$.

With this convention, we can state:

\begin{proposition}
	The map $g\mapsto C_g$ is an injective
	homomorphism of $(\K\Sigma)^\op$
	into $\Aut_{\Kalg}(\V)$.
	\qed
\end{proposition}

\section{Automorphisms of $\V$}
We are going to prove: 
\begin{theorem}\label{T:Aut-V}
	Assume that the integral domain $K$ has a reciprocal
	of $n\cdot1_K$ for some integer $n>1$.
	Then
	the map $g\mapsto C_g$ is an isomorphism 
	$$(\K\Sigma)^\op \simeq 
	\Aut_\Kalg(\V).$$
\end{theorem}

The hypothesis on $K$ is satisfied if $K$ has positive characteristic or is a field or contains the ring $\Z_p$
of $p$-adic integers  for some prime $p$.

This will immediately give us this description
of the Verde-Star-Riordan group:

\begin{corollary}
	If $n\cdot1_K$ is invertible in $K$ for some
	integer $n>1$, then
	$$\Hol_{\Kalg}(\V,^\times) \simeq 
	\V^\times\ltimes (\K\Sigma)^\op
	=\hatS\F^\times\ltimes (\K\Sigma)^\op
	,$$
with the product given by the usual formula
$$ (f,g) (f',g') = (f\cdot(f'\circ g), g'\circ g) $$
for $f,f'\in\V^\times$ and $g,g'\in\K$.
\end{corollary}

The proof proceeds by a series of lemmas.

\begin{lemma} Each $x_j$ is prime in $\F$.
\end{lemma}
\begin{proof}
	Without loss in generality, consider $j=d$.

	The map $\Phi:f(x_1,\ldots,x_d)\mapsto f(x_1,\ldots,x_{d-1},0)$
	is a $K$-algebra homomorphism $\F_d\to\F_{d-1}$. (We make the convention that
	$\F_0:=K$.)  We have
	$$ \ker\Phi = \{f\in\F_d: x_d|f\}.$$
	Suppose $x_d| ff'$. If $x_d\not| f$ and $x_d\not|f'$, then
	$$ \Phi(ff') = \Phi(f)\Phi(f') \not=0, $$
	since $\F_{d-1}$ is an integral domain. This is a contradiction.
	Thus $x_d|f$ or $x_d|f'$.
\end{proof}

\ignore{
\begin{lemma}
	Let $f,f'\in\F$ and $ff'=m\in S$. Then
	there exist $p,r\in S$ and $h\in\F^\times$ such that 
	$m=pr$,
	$f=ph$ and $f'=rh^{-1}$.
\end{lemma}
\begin{proof} This follows from the previous lemma by using
	induction on the degree of $m$.
\end{proof}
}

We abbreviate $\A:= \Aut_\Kalg(\V)$.

\begin{lemma}\label{L:aut-fixes-F-times}
	Suppose $n\cdot1_K$ is invertible in $K$ for some 
	integer $n>1$. 
	Let $\alpha\in\A$ and $f\in\F^\times$. Then
	$\alpha(f)\in\F^\times$.
\end{lemma}

\begin{proof}
	The binomial
	coefficients $\binom{1/n}{k}$ are well-defined in $K$
	for each nonnegative integer $k$ \cite{OF-IBC}.
	So if $h\in\M$
	the formula
	$$ r:= \sum_{k=0}^\infty \binom{1/n}{k} h^k $$
	defines an element of $\F^\times$ which is an $n$-th
	root of $1+h$: $r^n=1+h$.
	 
	 If the lemma were false, 
	 $v(\alpha(\F^\times))\subset \hatS$ would  not
	 be just $(1)$, so there would exist some $j$ such that 
	 the induced homomorphism
	 $\phi:\F^\times \to (\Z,+)$ defined by taking $f$ to the 
	 exponent of $x_j$ in $v(\alpha(f))$ is nontrivial,
	 and hence its image would be $m\Z$ for some nonzero
	 $m\in\Z$. Choose $f\in\F^\times$ with $\phi(f)=n$.
	 Replacing $f$ by $f/f_1$, we may assume $f_1=1$.
	 Then $f=1+h$ with $h\in\M$, and
	 choosing an $n$-th root $r$ as above, we would have
	 $n\phi(r)=\phi(f)=m$ and $0<\phi(r)<m$,
	 which is impossible.
\end{proof}

\begin{proof}[Proof of Theorem \ref{T:Aut-V}]
	Suppose $\alpha\in\A$.  By Lemma \ref{L:aut-fixes-F-times},
	$\alpha$ maps $\F^\times\to\F^\times$. Since $\alpha(1)=1$,
	$\alpha$ maps $\F\to\F$. Since the same is true of $\alpha^{-1}$,
	the restriction $\alpha|\F\in\Aut_\Kalg(\F)$.  By Theorem \ref{T:Aut-F},
	$\alpha$ coincides on $\F$ with composition with the map
	$$ g:=(\alpha(x_1),\ldots,\alpha(x_d)) \in\G. $$

	Consider $\alpha(x_j)$. Since $x_j\in\K^\times$, we have
	$\alpha(x_j)\in\K^\times$, so $\alpha(x_j)=m_jh_j$
	for some monomial $m_j\in\hatS$ and some $h_j\in\F^\times$.
	Since $x_j\in\F$, we have $m_j\in\F$.  Since $x_j$ is prime
	in $\F$, so is $m_j$, so it is equal to $x_r$ for some
	$r=\tau(j)$, say. Since $L(g)\in\GL(d,K)$, $\tau$ is a permutation,
	$\tau\in\Sigma_d$.  Thus $\tau^{-1}g\in\K$, and $g\in\Sigma\K=\K\Sigma$.
Thus the composition map $C_g$ is a $K$-algebra automorphism of the whole
of $\V$. 

It remains to show that $\alpha=C_g$. Fix $f\in\V$. Then
	$f=mr$, with $m=v(f)\in\hatS$ and
	$r\in \F$.
	Then the automorphism property gives $\alpha(m)=m\circ g$.
	Also $\alpha(r)=r\circ g$, since $r\in\F$, so we get
	$$\alpha(f) =\alpha(m)\alpha(r)= (p\circ g) \cdot (r\circ g) = f\circ g, $$
	as required.
\end{proof}

\subsection{Question}
We do not know whether Theorem \ref{T:Aut-V} remains true
in characteristic zero without the hypothesis on the existence of a reciprocal for some integer greater than one.  In particular, we do not know it for $K=\Z$. Could there be
some more exotic $\Z$-algebra automorphism (i.e. ring
automorphism) of $\V(\Z)$?  

The ring $\V(K)$ becomes a (complete, metrisable) topological ring if we
define a neighbourhood of $0$ to be a set
that contains some power $\M^k$ ($k\in\N$) of the maximal
ideal of $\F$. Lemma \ref{L:aut-fixes-F-times}
immediately yields the following ``automatic continuity''
result:

\begin{corollary}\label{C:automatic-continuity}
	If $n\cdot1_K\in K^\times$ for some integer $n>1$,
	then each $K$-algebra automorphism of $\V$ is
	continuous.
\qed\end{corollary}

We also have the following for all $K$:
\begin{proposition}
	If $\alpha$ is a continuous $K$-algebra automorphism
	of $\V$, then $\alpha\in\K\Sigma$.
\end{proposition}
\begin{proof}
We claim that $\alpha$ maps $\F$ into $\F$. 

For otherwise there is some $f\in\F$ with $\alpha(f)\not\in\F$,
so $h:=f-f_1\in\M$, and $\alpha(h)=\alpha(f)-f_1\not\in\F$.
For some $j\in\{1,\ldots,d\}$, $v(h)_j<1$. By continuity, since $\M$ is a neighborhood of $0$,
there exists some $N\in\N$ such that 
$\alpha(\M^N)\subset\M$.  But $h^N\in\M^N$, and 
$v(\alpha(h^N))_j=v(\alpha(f))_j^N<1$, a contradiction.

The rest of the proof proceeds just like that of
Theorem \ref{T:Aut-V}.
\end{proof}

\begin{corollary}
	For each $K$ and $d$, we have
	$$ \Hol_C(\V,^\times) \simeq \V^\times\ltimes(\K\Sigma)^\op $$
	where $C$ is the category of topological $K$-algebras
	and continuous $K$-algebra automorphisms.
	
\qed
\end{corollary}

Thus $\V^\times\ltimes(\K\Sigma)^\op$ is 
a holomorph with respect to a different category. We conjecture that
there are no exotic automorphisms 
in any case, and the two holomorphs coincide. 

\ignore{
An automorphism extends to tensor products $\Z_p\otimes_\Z\V$.
There is an $\M$-topology and an $\alpha(\M)$-topology.
}

\subsection{Distinguished subgroups}
By analogy with the terminology used for the
Riordan group, we define analogues 
of $\AP$ and $\LA$,
	the \emph{extended Appell subgroup} of $\Aut_\Kalg(\V)$
	$$ \EAP:= \EAP_d(K):= \{(f,\ONE):f\in\V^\times\},
	$$ 
	and the \emph{restricted Lagrange subgroup}
	$$ \ELA:= \ELA_d(K):= \{(1,g):g\in\K\Sigma\}.$$

We need to see what happens when an element of $\ELA$
acts on monomials.  First, for $m\in\hatS$ we define
\emph{the symmetric infimum}:
$$ \sinf (m) := \inf\{m\circ\tau:\tau\in\Sigma\}. $$
 
\begin{proposition}\label{P:monomial-circ-g}
	Let $m\in\hatS$ and $g\in\
\K\Sigma$. Let
$m\circ g= \sum_{p\in\hatS} f_pp$.
Then $f_p=0$ unless $p\ge\sinf(m)$.
\end{proposition}
\begin{proof} We can
	write $g=k\circ \lambda\circ\tau$ with $\tau\in\Sigma$, 
	$\lambda\in\K^{\times d}$ and $k\in\ker L|\K = \TI\cap\K$.
	Then 
	$$m\circ k= m(x_1(1+\HOT),\ldots,m_d(1+\HOT)),$$
	so $(m\circ k)_p=0$ unless $m\le p$. 
	Now $m\circ k\circ \lambda$ is obtained from
	$m\circ k$ by replacing each $x_j$ by a multiple
	$\lambda_jx_j$ with $\lambda_j\in K^\times$, so
	$(m\circ k\circ\lambda)_p=0$ if and only if
	$(m\circ k)_p=0$.  Finally, $(m\circ k\circ\lambda\circ\tau)_p$
	is nonzero if and only if $m\le p\circ\tau$,
	so only if 
	$p\ge\sinf(m)$.
	\end{proof}

\section{The Verde-Star matrix representation}
We define $M:\RVV\to \gl'(\hatS,K)$ by
$$ M(f,g)_{mn}:= ( f\cdot(n\circ g))_m, \ \forall m,n\in\hatS.$$
This is just another example of a matrix indexed by
a set, but you may also think of it as a `doubly-infinite'
matrix if you like, since the index set $\hatS$ has no
lower bound, unlike $S$.

We observe some obvious symmetries of $M(f,g)$:
\begin{proposition}
(1) $M(pf,g)_{mn} = M(f,g)_{m/p\,n}$ if $p\in\hatS$.
\\
(2) $M(f,\lambda\circ g)_{mn} = n(\lambda)M(f,g)_{mn}$ if $\lambda\in K^{\times d}$.
\\
(3)$M(f,\tau\circ g)_{mn}= M(f,g)_{m n\circ\tau}$ if
$\tau\in\Sigma$.
\qed
\end{proposition}

If we take $f\in\V^\times$ and  $g=x*h$ with $h\in(\F^\times)^d$,
then for $n,m\in\hatS$ we have
$$ (n\circ g) = n \cdot (n\circ h), $$
and if we abbreviate $v=v(f)$, then
$$ f\cdot(n\circ g) = v\cdot \left(\frac{f}{v}\right)
\cdot n\cdot (n\circ h), $$
so the entry $M(f,g)_{mn}$ is
$$ \left( vn\cdot \left(\frac{f}{v}\right) \cdot (n\circ h)\right)_m
=\left( (f/v) \cdot (n\circ h)\right)_{m/vn}.$$
Since $f/v\in\F$ and $n\circ g\in\F$, this entry
is zero unless $m/vn\ge1$.  Thus,
for fixed $m\in\hatS$, the entry is zero unless $n\le m/v$,
and for fixed $n\in\hatS$, the entry is zero
unless $m\ge vn$.
Hence $M$ maps $\V^\times\times\K$ into the set
$\widetilde{\gl}(\hatS, K)$ of matrices
$M\in\gl'(\hatS,K)$ such that
$$\begin{array}{rcl}
\forall m\in\hatS\ \exists R\in\hatS &:& n>R \implies M_{mn}=0,\\
\forall n\in\hatS\ \exists T\in\hatS &:& m<T \implies M_{mn}=0.
\end{array}.
$$
If $(f,g)$ is any element of $\RVV$, then $g=\tau\circ(x*f)$
for some $\tau\in\Sigma$ and $x*f\in\K$, and then
$M(f,g)_{mn}=M(f,x*f)_{mn\circ\tau}$ is zero unless
$v(f)n\circ\tau\le m$, so again we see that $M(f,g)$
belongs to $\widetilde\gl(\hatS,K)$.
 
The set $\widetilde{\gl}(\hatS,K)$ is closed under the 
usual matrix addition and multiplication, and forms
a $K$-algebra.  We denote its group of invertibles
by 
$$\widetilde{\GL}(\hatS,K):=  \widetilde{\gl}(\hatS,K)^\times.$$ 

The following is the Laurent series version of the Fundamental Theorem:
\begin{theorem}\label{T:3}
$M:\RVV\to\widetilde{\GL}(\hatS,K)$ is an injective
group homomorphism.
\end{theorem}
\begin{proof}
This is essentially the same proof as that of
Theorem \ref{T:1}, except that we have to 
be a little careful about the sums involved.
The key point is that ``intervals'' in $\hatS$,
i.e. sets of the form 
$$\{ m\in\hatS: n\le m\le p\}, $$
with $n,p\in\hatS$ are finite. If nonempty, i.e. if
$n\le p$, then regarded as sets in $\Z^d$ they are 
rectangular boxes of lattice points.  

Fix $(f,g),(f',g')\in \RVV$, and $m,n\in\hatS$. 
We need to show that
$$ M((f,g)(f',g'))_{mn}
=\sum_{p\in\hatS} M(f,g)_{mp}
M(f',g')_{pn}. $$
Once we have this, the remainder of the proof
goes as before.

Let $v=v(f)$, $v'=v(f')$.

For $p,r,s\in\hatS$, consider the element 
$ a:= f_rf'_s(p\circ g)_{m/r}(n\circ g)_{p/s}\in K.$
This is nonzero only if
$$ r\ge v, s\ge v', \frac mr\ge p \textup{ and } \frac ps \ge n, $$
and hence only if 
$$ nv'\le ns\le p \le \frac mr  \le \frac mv, $$
so only for a finite set of $p$. Also, for each $p$,
the term $a$ is nonzero only for
$v\le r\le m/p$ and $v'\le s\le p/n$,
i.e. for a finite set of pairs $(r,s)$.
Thus the sum
$$ \sum_{p,r,s\in\hatS} 
f_rf'_s(p\circ g)_{m/r}(n\circ g')_{p/s}
$$
has only a finite number of nonzero terms, and 
can be summed in any order to the same value.
We find that
$$\sum_{p\in\hatS} M(f,g)_{mp}
M(f',g')_{pn}
= \sum_p\sum_r\sum_sf_rf'_s(p\circ g)_{m/r}(n\circ g')_{p/s},$$
and rearranging, this equals
$$ \sum_r f_r \sum_s f'_s \sum_p (p\circ g)_{m/r}(n\circ g')_{p/s}. $$
The inner sum is equal to 
$$(s \cdot n\circ g')\circ g)_{m/r}
=(r\cdot s\circ g \cdot n\circ g'\circ g)_m$$
so adding up, the whole is
$$ (f \cdot f'\circ g \cdot n\circ g'\circ g)_m = M((f,g)(f',g'))_{mn},
$$
as required.
\end{proof}

\subsection{Another description}
We can view the Verde-Star-Riordan group 
as a subgroup of the automorphism group of $\V$
in the category $\Kmod$ of $K$-modules.

For $f\in\V$, we define the \emph{multiplier map}
$M_f:\V\to\V$ by 
$$ M_f(f') = f\cdot f',\ \forall f'\in\V. $$  
The map $M:f\mapsto M_f$ is an injection of
$\V$ into the $K$-algebra of $K$-module
endomorphisms of $\V$.  It maps
invertible elements $f\in\V^\times$ to
$K$-module automorphisms of $\V$.

We have seen that for $g\in\K\Sigma$ the composition
map $C_g$ is a $K$-algebra automorphism of $\V$,
so, \textit{a fortiori,} it is a $K$-module 
automorphism.

The topological-K-algebra holomorph 
$\V^\times\ltimes(\K\Sigma)^{\op}$ is mapped
injectively into $\Aut_{\Kmod}(\V)$ by the map
$$ (f,g) \mapsto M_f\circ C_g. $$

\section{Reversibility}
The work of Luzon et al \cite{LMP} completed the explicit description
of the reversible elements of $\RI_1(\C)$. We would like to
see similarly explicit descriptions of $\RI_d(K)$
for all $d\in\N$ and each integral domain with identity.

The following basic result translates the problem into
\lq nuts and bolts\rq:

\begin{proposition}
	Fix $d\in\N$ and an integral domain $K$ with identity,
	and abbreviate $\RI=\RI_d(K)$.
	Let
	$(a,b), (f,g)\in\RI$. 
	Then $(a,b)$ reverses $(f,g)$ if and only if
	$b$ is a fixed point of $y\mapsto gyg$ and
	$a$ is a fixed point of 
	$$ x\mapsto g(f)\cdot b(x)\cdot b(b(f)). $$
\end{proposition}
(The dot here represents the multiplication in $\F_d$.)
\begin{proof}
	Directly from the definition one deduces that
	$(a,b)$ reverses $(f,g)$ if and only if
	$$ (f,g)(a,b)(f,g)=(a,b). $$
	The left hand side evaluates to
	$$\left(g(f)\cdot b(a) \cdot (b^2)(f),gbg\right). $$
\end{proof}

Obviously, this is just a special case of something
that holds in any semidirect product.  We would like 
a completely explicit classification of the $(f,g)\in\RI$
for which there exists a reverser $(a,b)$.  There is
a good deal to be done:

\begin{itemize}
	\item  The reversibility problem for $\RI_d(K)$ 
		is completely solved
		only for $d=1$ and $K=\C$ \cite{LMP}.
	\item The first step would be the explicit
		description of the reversibles 
		in the Lagrange subgroup, isomorphic
		to $\G$.  This has been completed
		for $d=1$ and fields $K$ of characteristic zero \cite{OFS},
		but remains open for some other
		integral domains and is completely
		open in positive characteristic.
	\item  For $d>1$, only the so-called \textit{generic}
		reversibles of $\G_d(\C)$ have been described \cite{OFZ},
		and it remains to deal with the rest and with
		other rings $K$. For instance, we do not
		know which $g\in\G_2$, tangent to the identity,
		are reversible in $\G_2(K)$ for any ring $K$ of
		characteristic zero.
\end{itemize}

As Polya said, for every unsolved problem there is always an easier problem
that you can't solve, and in this case it may be useful to
start by asking which elements of $\G$, tangent to the identity,
are reversible in the whole Riordan group. Also, one could hope
that the matrix representation of our main result will help.

\section{Other Riordan Group Constructions}
\subsection{} Cheon and collaborators have introduced other higher-dimensional 
kinds of Riordan groups.  A referee suggested we relate them
to our work.

The Riordan groups $\RI_d(K)$ we have just defined in this paper are constructed from
structures composed of formal maps and series in \textit{several variables},
and their elements are best represented as \textit{two-dimensional arrays}
where the rows and columns are indexed by the semigroup of $d$-dimensional
monic monomials.  It would be possible \textit{but not particularly helpful} to
represent them as multidimensional arrays indexed by integers.  
The labelling by these semigroup
elements preserves information essential to the group structure, and
enables simple and transparent formulas.

\subsection{}
The \textit{multidimensional Riordan arrays} introduced and applied by Cheon and Jin
\cite{CJ} 
are obtained from a group of formally-invertible formal maps in \textit{one variable}
by a finite number of repeated semidirect products with an abelian group
of convolution-product-invertible formal series in \textit{one variable}, and
the elements of their groups are representable as \textit{multidimensional arrays} 
(indexed by integers). They work over the ring $K=\C$ of complex numbers,
but the construction is valid for any ring.
These Cheon-Jin groups relate to the composition action of what we call $\G_1=\G_1(\C)$
on the group 
$(\F_1^\times)^n$, 
of $n$-tuples of elements of 
$\F_1^\times$, 
and also on the space $(\F_1)^{\N}$
of infinite sequences of power series in one variable.  These groups
are unlike our $\RI_d$ for $d>1$, but the idea could be used to define
corresponding extensions $(\F_d^\times)^n\ltimes \G_d$.

\subsection{}
In another direction, 
Cheon, Huang and Kim \cite{CHK} discuss \textit{multivariate Riordan arrays}.
They work in the context of a power series ring of Krull dimension $d$ over
a field $K$.  They explain that such a ring can be identified (in various ways) with
what we call $\F_d(K)$ by selecting a \textit{set of variables} $\textbf{Z}$,
i.e. a $d$-tuple of elements that generate the maximal ideal.  
We observe that the collection of possible
`sets of variables' is, in essence and ignoring order, the same as our $\G_d(K)$, because
composition with a formal map $g\in\G_d$ turns one `set of variables'
into another.
Starting with the idea of a Schauder basis for the ring (a sequence $(b_j)$ such
that every element of the ring is an infinite sum 
$\sum_{j=1}^\infty \lambda_j b_j$, converging in the natural topology,
with unique coefficients $\lambda_j\in K$) they
single out a special kind, which they call a \textit{Riordan-Schauder basis}.
The key ingredient here is the imposition of some linear
order on the set of monomials (our $S_d$), putting them into
a sequence.  We observe that a general Riordan-Schauder
basis for the ring $\F_d$ is in fact what you get when you compose
any such  sequence of monomials with any $g\in\G_d$.
With this setup, they can define a \textit{Lagrange group with unit} $\textbf{Z}$
as a collection of `sets of variables', 
and a \textit{Riordan monoid with respect to $\textbf{Z}$} as a 
set of products,
containing a \textit{Riordan group with respect to } $\textbf{Z}$ with an 
corresponding \textit{Appell subgroup}.  They represent the elements of
this group by two-by-two matrices of elements of $K$, using the
Riordan-Schauder basis, and they call these matrices \textit{Riordan arrays}.  
They give a version of the Fundamental Theorem.
They have to be careful about the character of the order imposed on
the monomials.  Apart from details, it seems that
for each choice of units and basis order, the groups
of Riordan arrays they construct are in fact isomorphic to our
$\RI_d(K)$, because each is the semidirect product of groups
isomorphic to $\F_d^\times$ and $\G_d$.
It would be more accurate to say that the Cheon-Huang-Kim
groups are isomorphic to $\Hol_C(\F_d,^\times)$, where
$C$ is the category of \textit{topological} $K$-algebras
and continuous $K$-algebra homomorphisms.
These are isomorphic to $\RI_d$ because of the automatic
continuity of $K$-algebra automorphisms of $\F_d$,
when $\F_d$ is equipped with the natural topology.

\subsection{}
We suggest that the focus on the order of the monomials or
on the order of the Riordan-Schauder basis is perhaps a distraction
from the essentials in the study of the Riordan
groups in higher dimensions. The essential thing to focus
on is the action on monomials, and the semigroup structure
of $S_d$.  This is not to say that for a specific application
it may not be useful to order the monomials.

\subsection{} 
The Fundamental Theorem in Cheon-Huang-Kim is an immediate consequence
of the facts that the elements of the Riordan monoid determine
\textit{continuous} linear endomorphims of a topological vector
space over $K$ that has a Schauder basis, and the topology is
the inductive limit of the discrete topologies on the
spans of finite subsets of the basis. Noting this,
one remarks that in our context the fundamental equation \ref{E:2} can also be proven
by factoring our map $M$ of Section 3 through the space
of continuous $K$-linear endomorphisms of $\F_d$, and
using the graded basis of monomials, without worrying about
the order of the monomials of each degree.

\subsection{}
We close with a few remarks on functoriality.

	$\Hol_C$ is not, in general, a functor from $C$ to the category
	of groups. However, if we define
	$C^\iso$ to be the category having the same objects
	as $C$ but only $C$-invertible arrows $A\to B$
	as arrows, then \\
	(1) $C\mapsto C^\iso$ is a functor on the category
	of categories,\\
	(2) the identity map $C \to C$ is a functor
	from $C^\iso$ to $C$,
	and\\
	(3) the restriction $\Hol_{C^\iso}$ of $\Hol_C$ to
	$C^\iso$ is a functor from $C^\iso$ to the category of groups.

\section*{Acknowledgement}
	I am grateful to John Murray for useful conversation on the subject of this paper,
	and to the referee for a careful reading, helpful suggestions and corrections.

\end{document}